\documentclass[12pt,reqno]{amsart}
\usepackage{hyperref}

\def\m{\mathfrak{m}}
\def\R{\mathcal{R}}
\def\NN{\mathbb{N}}
\def\red{\operatorname{red}}
\def\jr{\operatorname{jr}}
\newtheorem{theorem}{Theorem}[section]
\newtheorem{definition}[theorem]{Definition}

\newtheorem{lemma}[theorem]{Lemma}
\newtheorem{proposition}[theorem]{Proposition}
\newtheorem{example}[theorem]{Example}

\newtheorem{question}[theorem]{Question}
\newtheorem{remark}[theorem]{Remark}
\newtheorem{corollary}[theorem]{Corollary}

\begin{document}
\title{A Northcott type inequality for Buchsbaum-Rim coefficients}

% author  information
\author[Balakrishnan R.]{Balakrishnan R.$^*$}
\email{r.balkrishnan@gmail.com}

\author{A. V. Jayanthan}
\email{jayanav@iitm.ac.in}
\address{Department of Mathematics, Indian Institute of Technology
Madras, Chennai, INDIA -- 600036.}
\thanks{$*$ Supported by the Council of Scientific and Industrial
Research (CSIR), India}
\thanks{AMS Classification 2010: 13D40, 13A30}
\maketitle

\begin{abstract}

In 1960, D.G. Northcott proved that if $e_0(I)$ and $e_1(I)$ denote
zeroth and first Hilbert-Samuel coefficients of an $\mathfrak
m$-primary ideal $I$ in a Cohen-Macaulay local ring $(R,\mathfrak m)$, then
$e_0(I)-e_1(I)\le \ell (R/I)$. In this article, we study an analogue of
this inequality for Buchsbaum-Rim coefficients. We prove that if
$(R,\mathfrak m)$ is a two dimensional Cohen-Macaulay local ring and $M$ is
a finitely generated $R$-module contained in a free module $F$ with
finite co-length, then $br_0(M)-br_1(M)\le \ell (F/M)$, where
$br_0(M)$ and $br_1(M$) denote zeroth and first Buchsbaum-Rim
coefficients respectively.
\end{abstract}

\section{Introduction}

Let $(R,\mathfrak m)$ be a Noetherian local ring of dimension
$d>0$.
Let $M \subset F=R^r$ be a finitely generated $R$-module
such that $\ell (F/M) < \infty$, where $\ell(-)$ denote the length
function. Let
$\mathcal S(F)=\underset {n\ge 0}{\bigoplus}{\mathcal S_n(F)}$ denote
the Symmetric algebra of $F$,  and $\mathcal R(M)=\underset {n\ge
0}{\bigoplus}\mathcal R_n(M)$ denote the Rees algebra of $M$, which is
image of the natural map from the Symmetric algebra of $M$ to the
Symmetric algebra of $F$. Generalizing the notion of Hilbert-Samuel
function, D. A. Buchsbaum and D. S. Rim studied the function
$BF(n)=\ell (\mathcal S_n(F)/\mathcal R_n(M))$ for $n\in \mathbb{N}$. In
\cite {BR}, they proved that $BF(n)$ is
given by a polynomial of degree $d+r-1$ for $n\gg0$, i.e., there
exists a polynomial $BP(x)\in \mathbb Q[x]$ such that $BF(n)=BP(n)$
for $n\gg 0.$ The function $BF(n)$ is called the Buchsbaum-Rim
function of $M$ with respect to $F$ and the polynomial $BP(n)$ is
called the corresponding Buchsbaum-Rim polynomial. Following the
notation used for the Hilbert-Samuel polynomial, one writes the
Buchsbaum-Rim polynomial as
$$BP_M(n)=\sum_{i=0}^{d+r-1}(-1)^i br_i(M){n+d+r-i-2\choose d+r-i-1}.$$
The coefficients $br_i(M)$ for $i=0,\ldots ,d+r-1$ are known as
Buchsbaum-Rim coefficients.

When $r=1$, set $M = I$, an $\m$-primary ideal in $R$. In this case,
Buchsbaum-Rim polynomial coincides with usual Hilbert-Samuel
polynomial and its coefficients will be denoted by $e_i(I)$,
called the Hilbert-Samuel coefficients. While the Hilbert-Samuel
coefficients are very well studied objects and the relationship of its
properties with the properties of the ideal and the corresponding
blowup algebras are well known, there is a dearth of results in this
direction on Buchsbaum-Rim coefficients.
In \cite{DGN}, D. G. Northcott proved that
\begin{theorem}\cite[Theorem 1, 3]{DGN}\label{n}
Let $(R,\mathfrak m)$ be a Cohen-Macaulay local ring of dimension $d >
0$ with infinite residue field and let $I$ be an $\mathfrak m$-primary
ideal. Then
  \begin{enumerate}
    \item $e_0 (I)-e_1 (I) \leq \ell {(R/I)}.$
    \item $e_1 (I)\ge 0$ and the equality holds if and only if $I$ is generated by $d$ elements(i.e., $I$ is a parameter ideal).
  \end{enumerate}
\end{theorem}
C. Huneke and A. Ooishi independently studied the equality in Theorem
\ref{n}(1):
\begin{theorem}\label{ho}(\cite{Hun},\cite{o})
Let $(R,\m)$ be a Cohen-Macaulay local ring of dimension $d>0$ and
let $I$ be an $\m$-primary ideal of $R$. Then $e_0(I) - e_1(I) =
\ell(R/I)$ if and only if there exists a minimal reduction $J \subset
I$ such that $I^2 = JI$.
\end{theorem}

In \cite{BUV}, J. Brennan, B. Ulrich and W. V. Vasconcelos proved that
Theorem \ref{n}(2) generalizes to Buchsbaum-Rim coefficient: if
$(R,\m)$ is a Cohen-Macaulay ring, then $br_1(M)$ is non-negative and
$br_1(M)$ vanishes if and only if $M$ is a parameter module. In
\cite{HH2}, F. Hayasaka and E. Hyry studied the Buchsbaum-Rim function
of a parameter module $N$ over a Noetherian local ring and they proved
that $br_1(N)\le 0$ and equality holds if and only if the ring is Cohen-Macaulay.

Motivated by Theorems \ref{n} and \ref{ho}, we ask:
\begin{question}
Let $(R,\m)$ be a Cohen-Macaulay local ring of dimension $d >0$, $F$
be a free module of rank $r$ and $M$ be a submodule such that
$\ell(F/M) < \infty$. Then is the inequality $br_0(M) - br_1(M) \leq
\ell(F/M)$ true? Is it true that the equality holds if and only if the
reduction number of $M$ with respect to a minimal reduction is at most
one?
\end{question}

In this article, we prove the inequality in the case $\dim R = 2$ and
show that the module having reduction number one is a sufficient
condition for equality. We now give a short description of the paper.

In Section 2, we begin
with an example to show that the Northcott type inequality does not
hold true for Buchsbaum-Rim coefficients if $\dim R = 1$. We then
consider the case $\dim R = d \geq 2$ and $M = I_1 \oplus \cdots
\oplus I_r \subset R^r$, where $I_i$'s are $\m$-primary ideals in $R$.
When the Rees algebra $\R(M)$ is Cohen-Macaulay, we obtain an
expression for the Buchsbaum-Rim coefficients $br_0(M)$ and $br_1(M)$
in terms of the mixed multiplicities of the ideals $I_1, \ldots, I_r$
and derive that if $d = 2$ and $r = 2$, we have the equality $br_0(M)
- br_1(M) = \ell(F/M)$. We also prove that if $\dim R = 2$ and $M$ is
an $R$-submodule of $F = R^r$ with reduction number of $M$ being one,
then $br_0(M) - br_1(M) = \ell(F/M)$.

In Section 3, we define an analogue of Sally module of a module with
respect to a reduction. We obtain an expression for the Hilbert
polynomial of the Sally module using the Buchsbaum-Rim coefficients
and derive the inequality $br_0(M) - br_1(M) \leq \ell(F/M)$ when
$\dim R = 2$. We also prove that if $\red(M) = 1$, then the equality
holds, Theorem \ref{northcott}.

In Section 4, we study the problem for modules which are direct sum of several copies of an $\m$-primary ideal. Let $(R,\m)$ be a Cohen-Macaulay local ring of
dimension $d \geq 2$ and
$I$ be an $\m$-primary ideal. Let $M = I \oplus \cdots
\oplus I$ ($r$-times, $r \geq 1$), then $br_0(M) - br_1(M) \leq
\ell(F/M)$, Theorem \ref{newthm}. We also prove that in dimension $2$,
the equality holds if and only if $\red(M) = 1$, Corollary
\ref{directsum}. We also compute some examples to illustrate the
Northcott inequality.
\vskip 2mm
\noindent
\textbf{Acknowledgement:} The authors would like to thank E. Hyry, S.
Zarzuela and W. V. Vasconcelos for going through a first draft and
making many useful suggestions for improvement and further research.
We are also thankful to the referee for a meticulous reading and
suggesting several improvements.
\section{Reduction number one}

In this section, we obtain certain sufficient conditions for the
equality $br_0(M) - br_1(M) = \ell(F/M)$.  We begin by recalling some
basic terminologies which are essential for studying Buchsbaum-Rim
polynomial.  Let $M\subseteq F = R^r$ be such that $\ell(F/M) <
\infty$. Let  $N$ be a submodule of $M$. We say that $N$ is reduction
of $M$ if Rees algebra $\mathcal R(M)$ is integral over the
$R$-subalgebra $\mathcal{R}(N)$.  Equivalently this condition is
expressed as $\mathcal R_{n+1}(M)=N\mathcal R_n(M)$ for $n\gg 0$,
where the multiplication is done as $R$-submodules of $\R(M)$. The
least integer $s$ such that $\mathcal R_{s+1}(M)=N\mathcal R_s(M)$ is
called the reduction number of $M$ with respect to $N$, denoted as
$\red_N(M)$. The reduction number of the module $M$, denoted $\red
(M),$ is defined as $\red (M)=\min \{\red_N(M):~ N \mbox{ is a minimal
reduction of M}\}$. If $N$ is a submodule of $F$ generated by $d+r-1$
elements such that $\ell(F/N) < \infty$, then $N$ is said to be a
parameter module. It was proved in \cite{BUV} that if $\ell(F/M) <
\infty$, then there exists minimal reduction generated by $d+r-1$
elements.  For more details on minimal reductions, we refer the reader
to \cite{hs} and \cite{WV}.

In the following example, we show that, for $1$-dimensional
Cohen-Macaulay local rings, the Northcott type inequality does not
hold for Buchsbaum-Rim coefficients.
\begin{example}
Let $R=k[\![X,Y\!]]/(X^2)$ and $I=(x,y)$, where $x=\overline{X}$ and
$y=\overline{Y}$, and $k$ is a field. Then $R$ is a $1$-dimensional
Cohen-Macaulay local ring.
It can be seen that $\ell(R/I^n) = \ell(k[\![X,Y]\!]/(X^2,
(X,Y)^n)) = 2n - 1$. Therefore, $e_0 = 2$ and $e_1=1$.

Let $F=R\oplus R$ and $M=I\oplus I$. Then it follows from \cite[Theorem
2.5.2]{R} that the Buchsbaum-Rim polynomial of $M$ is given by
\begin{eqnarray*}
    BP(n)& =& [e_0 n- e_1]{n+1\choose 1}
    =2e_0{n+1\choose 2}- e_1{n\choose 1}- e_1 \\
	& = & 4 {n+1 \choose 2} - {n \choose 1} - 1.
\end{eqnarray*}

Hence we have $br_0(M) =4$ and $br_1(M) =1$. Therefore
$$
    br_0-br_1= 3 > 2 = \ell(F/M).
$$
\end{example}

Now we study the Buchsbaum-Rim polynomial of a special class of
modules, namely a direct sum of $\m$-primary ideals in a
Cohen-Macaulay local ring.
Let $(R,\m)$ be a $d$-dimensional Noetherian local ring and $\textbf
{I} = I_1,\ldots, I_r$ be a sequence of $\m$-primary ideals. For
$\underline{u} =
(u_1,\ldots , u_r)\in \mathbb{N}^r,$ let
$\textbf{I}^{\underline{u}}=I_1^{u_1}\cdots I_r^{u_r}$. Then $\ell
(R/\textbf{I}^{\underline{u}})$ is given by a polynomial
$P(\underline{u})$ in $r$ variables
of total degree $d$ for $u_i\gg 0$ for each $i$, \cite{bhat}. Write the Bhattacharya
polynomial of $\textbf{I}$ as
$$P_{\textbf{I}}(\underline{u})=\underset { { \alpha} \in \mathbb{N}^r,|\alpha|\le
d}\sum e_{\alpha}(\mathbf{I}){u_1 \choose {\alpha}_1}\cdots {u_r \choose {\alpha}_r}.$$
Here $e_{\alpha}(\bf I)$ with $|\alpha|=d$ are known as the mixed
multiplicities of $I_1, \ldots, I_r$.

For $i=0,\ldots,d$, set $E_i=\underset {\alpha \in \mathbb {N}^r,
|\alpha|=i}\sum e_{\alpha}(\bf I)$. Below, we obtain an expression for
the Buchsbaum-Rim multiplicity and the first Buchsbaum-Rim coefficient
in terms of the Bhattacharya coefficients.

\begin {proposition}\label{br-formula}
Let $(R,\m)$ be $d$-dimensional Cohen-Macaulay local ring, $I_1,
\ldots, I_r$ be $\m$-primary ideals and $M=I_1\oplus \cdots \oplus
I_r\subset R^r$.  If $\ell(R/\textbf{I}^{\underline{u}}) =
P_{\textbf{I}}(\underline{u})$ for
all $\underline{u} \in \NN^r$, then $br_0(M) = E_d$ and
$br_1(M)=(d-1)E_{d}-E_{d-1}.$
\end {proposition}
\begin{proof}
Let $BP(n)$ denote the Buchsbaum-Rim polynomial corresponding to the
function $BF(n)=\ell (\mathcal S_n(F)/\R_n(M))$. First note that
$\mathcal S(F)\cong R[t_1,\ldots,t_r]$ and $\R (M) \cong
R[I_1t_1,\ldots ,I_rt_r]$, where $t_1,\ldots, t_r$ are indeterminates
over $R$. Therefore $BF(n)= \underset{\underline{u}\in
\mathbb{N}^r,|\underline{u}|=n}\sum
\ell(R/\textbf{I}^{\underline{u}})$. Hence for all $n\in \mathbb{N}$ we have
  \begin{eqnarray*}
    BP(n)&=&BF(n)\\
	&=&\underset{\underline{u}\in \mathbb{N}^r,|\underline{u}|=n}\sum
	P_{\textbf{I}}(\underline{u})\\
	&=&{\underset{\underline{u}\in \mathbb{N}^r,|\underline{u}|=n}\sum} ~
	{\underset{\underline{\alpha}\in
	\mathbb{N}^r,|\underline{\alpha}|\le d}\sum }
	e_{\underline{\alpha}}(\textbf{I}){u_1 \choose {\alpha}_1}\cdots {u_r \choose {\alpha}_r}\\
	&=&\underset{\underline{\alpha}\in
	\mathbb{N}^r,|\underline{\alpha}|\le d}\sum
	e_{\underline{\alpha}}(\textbf{I}) \underset{\underline{u}\in
	\mathbb{N}^r,|\underline{u}|=n}\sum {u_1 \choose {\alpha}_1}\cdots {u_r \choose {\alpha}_r}\\
	&=&\underset{\underline{\alpha}\in
	\mathbb{N}^r,|\underline{\alpha}|\le d}\sum
	e_{\underline{\alpha}}(\textbf{I}){n+r-1\choose
	  |\underline{\alpha}|+r-1}\\
         &=&E_d{n+r-1\choose d+r-1}+E_{d-1}{n+r-1\choose d+r-2}+\cdots
  \end{eqnarray*}
  By using Pascal's identity repeatedly, we observe that
  $${n+r-1\choose d+r-1}={n+d+r-2\choose d+r-1}-\left[{n+d+r-3\choose
  d+r-2}+\cdots+{n+r-1\choose d+r-2}\right].$$
  Hence $BP(n)=E_d{n+d+r-2\choose d+r-1}+[E_{d-1}-(d-1)E_d]{n+d+r-3\choose d+r-2}+\cdots$.
  It follows that $br_0(M)=E_d$ and $br_1(M)=(d-1)E_d-E_{d-1}$.
\end{proof}
Note that if the $\R(M)$ Cohen-Macaulay, then by \cite[Theorem
6.1]{hyry}, $\ell(R/\textbf{I}^{\underline{u}}) = P_\textbf{I}(\underline{u})$ for
all $\underline{u} \in
\NN^r$ and hence $BF(n) = BP(n)$ for all $n \geq 0$.
As a consequence we obtain the equality $br_0(M) -
br_1(M) = \ell(F/M)$:
\begin{corollary}\label{br-equality1}
Let $(R,\mathfrak m)$ be a $2$-dimensional Cohen-Macaulay local ring
with infinite residue field. Let $I$ and $J$ be $\m$-primary ideals in
$R$. Let and $M = I \oplus J \subset R \oplus R$. If
$\mathcal R(M)$ is Cohen-Macaulay, then
$br_0(M)-br_1(M) = \ell (F/M)$.
\end{corollary}
\begin{proof}
By applying previous proposition with $d=2$ and $r=2$, we get
$br_0(M)-br_1(M)=E_2 - (E_2 - E_1) = E_1 =e_{10}+e_{01}$. Since
$\R(M)$ is Cohen-Macaulay, it follows from \cite[Theorem 6.3]{JV} that $e_{10} =
\ell(R/I)$ and $e_{01} = \ell(R/J)$. Therefore,
$br_0(M) - br_1(M) =\ell (R/I)+\ell (R/J)=\ell (F/M).$
\end{proof}

Note that the above Theorem can also be derived from Theorem
\ref{br-equality2}. We have provided the above proof as it is
independent and involves a different technique.
\begin{remark}\label{red}
Let $(R,\m)$ be a two dimensional Cohen-Macaulay local ring, $I_1,
\ldots, I_r$ be $\m$-primary ideals and $M = I_1 \oplus \cdots \oplus
I_r$. Let $\jr(I_i|I_j)$ denote the joint reduction number of $I_i$
and $I_j$ (we refer the reader to \cite{hyry1} and \cite{jkv} for
definition and some basic results concerning joint reductions).
It is proved in \cite[Corollary 4.5]{SUV} that if $\jr(I_i|I_j) = 0$
for any $i, j \in \{1, \ldots, r\}$, then $\R(M)$ is Cohen-Macaulay.
We would like to observe here that the converse is also true. Suppose
$\R(M)$ is Cohen-Macaulay. Then a modification of \cite[Theorem
6.1]{liu} gives that $\R(I_{i_1} \oplus \cdots \oplus I_{i_s})$ is
Cohen-Macaulay for any $\{i_1, \ldots, i_s\} \subset \{1, \ldots,
r\}$. In particular, $\R(I_i)$ is Cohen-Macaulay for each $i = 1,
\ldots, r$ and $\R(I_i \oplus I_j)$ is Cohen-Macaulay for $\{i,j\}
\subset\{1, \ldots, r\}$. This implies that $\jr(I_i|I_j) = 0$ for any
$1 \leq i, j \leq r$.
\end{remark}

In the following example, we compute the Buchsbaum-Rim coefficients.
\begin{example}
  Let $R=k[[X,Y]], I=\mathfrak m=(X,Y), J=(X^2,Y)$. Then
  $red(I)=red(J)=0$. Also $(Y)I+(X)J=IJ$ implying
  $\jr(I|J)=0$ so that the Rees algebra $R(I, J) \cong \R(I\oplus
  J)$ is Cohen-Macaulay by \cite[Theorem 6.3]{JV}.
  Set $F=R\oplus R$ and $M=I\oplus J.$ Therefore, we have
  $BF(n)=BP(n)$ for all $n$. Using any of the computational
  commutative algebra packages, it can be seen that
  $\ell (\mathcal S_1(F)/\mathcal R_1(M))=3, \ell (\mathcal S_2(F)/\mathcal R_2(M))=13, \ell (\mathcal S_3(F)/\mathcal R_3(M))=34, \ell (\mathcal S_4(F)/\mathcal R_4(M))=70.$
  In turn, we get the Buchsbaum-Rim polynomial as $BP(n)=4{n+2\choose
  3}-1{n+1\choose 2}$. Hence $br_0(M)-br_1(M)=4-1=3=\ell(F/M).$
\end{example}

D. Katz and V. Kodiyalam studied the Cohen-Macaulayness of the Rees
algebra of modules over two dimensional regular local rings. They
proved:
\begin{theorem}\cite[Corollary 4.2]{KK} \label{kk-thm}
  Let $(R,\mathfrak m)$ be a two dimensional regular local ring and M
  be a finitely generated torsion free R-module, then the following
  are equivalent:
  \begin{enumerate}
    \item $NM=\mathcal R_{2}(M)$ for every minimal reduction $N\subset
	  M$;
    \item The Rees algebra $\mathcal R(M)$ is Cohen-Macaulay;
    \item $\ell (\mathcal S_{n+1}(F)/\mathcal R_{n+1}(M))=
	  br_0(M){n+r+1\choose r+1}- \ell (M/N){n+r\choose r}$ for all
	  $n\geq 0$ and every minimal reduction $N\subset M$.
  \end{enumerate}
\end{theorem}

Since $N$ is a parameter module and a minimal reduction of $M$, $br_0(M)=br_0(N)=\ell (F/N)$,
\cite[Theorem 3.1]{BUV}. Hence in this case $br_0(M)-br_1(M)=\ell (F/N)-\ell (M/N)=\ell (F/M)$.
A. Simis, B. Ulrich and W. V. Vasconcelos proved
that if $(R,\m)$ is a two dimensional Cohen-Macaulay local ring and $M
\subset F=R^r$ is a module with $\ell(F/M) < \infty$, then $\R(M)$ is
Cohen-Macaulay if and only if $\red(M) \leq 1$, \cite[Proposition 4.4]{SUV}. By adopting the proof of
Katz and Kodiyalam, we prove (1)
implies (3) of the above theorem in the case of $2$-dimensional of Cohen-Macaulay rings.
Though the proof works on the same lines, the two isomorphisms used in
the proof are justified by a result of F. Hayasaka and E. Hyry. We recall the
result from \cite {HH}.
For an $R$-module $M$, let $\widetilde{M}$ denote the matrix whose
columns correspond to the generators of $M$ with respect to a fixed
basis of $F$. The matrix $\widetilde{M}$ is said to be perfect if the
zeroth Fitting ideal of $M$ is a proper ideal with maximal grade.

\begin{theorem}\cite [Theorem 4.4] {HH}\label{hyry-perfect}
  Let $R$ be a Noetherian ring and $F$ an $R$-free module of rank $r>0$. Let $M$ be a submodule of $F$ such that $\widetilde{M}$ is a perfect matrix of size $r\times (r+1)$. Then the natural surjective homomorphism
  $$\phi_{1}: (F/M)[Y_1,\ldots, Y_{r+1}] \rightarrow  G_1(M)$$
  is an isomorphism, where $G_1(M)=F\mathcal R(M)/\mathcal R(M)^{+}.$
  \\In particular the $R$-module $F\mathcal R_n(M)/\mathcal R_{n+1}(M)$ is a direct sum of ${n+r\choose r}$ copies of $F/M$.
\end{theorem}

\begin{remark}
It is known that if $M$ is a parameter module, then the matrix
$\widetilde{M}$ is perfect, \cite{HH}. So in particular, when the ring $R$ is a two
dimensional Cohen-Macaulay local ring and $M$ is a parameter module,
above theorem is true, \cite[Corollary 4.5]{HH}.
\end{remark}
\begin{lemma}\label {lemmaKK}
Let $(R,\mathfrak m)$ be a two dimensional Cohen-Macaulay local ring with
infinite residue field and $M\subset F=R^r$ be a finitely generated
$R$-module with $\ell(F/M) < \infty$. Let $N\subset M$ be a minimal
reduction generated by $\{c_1, \ldots, c_{r+1}\}$.
If $k={n+r\choose r}$ and $\phi: F^k \rightarrow F\mathcal
R_n(N)$ be the surjective $R$-module homomorphism defined by
$\phi(f_1,\ldots, f_k) =
\underset {i_1+\cdots +i_{r+1}=n}{\underset {i=1} {\overset {k}
\sum}} f_i c_1^{i_1}c_2^{i_2}\cdots c_{r+1}^{i_{r+1}}$, then the
corresponding induced maps
    $$\phi_1: \left(\frac {F}{N}\right)^k\rightarrow \frac {F\mathcal
	R_n(N)}{\mathcal R_{n+1}(N)}\hspace*{0.7cm} and \hspace*{0.7cm}
      \phi_2: \left(\frac {F}{M}\right)^k\rightarrow \frac {F\mathcal R_n(N)}{M\mathcal R_n(N)} $$
      are isomorphisms.
\end{lemma}

\begin{proof} It follows from the previous remark that
$\phi_1$ is an isomorphism. Surjectivity of $\phi_2$ is clear. For an
element $f \in F$, let $\bar{f}$ denote its image in $F/M$ and
$\widetilde{f}$ denote its image in $F/N$.
Suppose $\phi_2(\bar{f_1},\ldots , \bar{f_k})=0$. This implies
$$\underset {i_1+\cdots +i_{r+1}=n}{\underset {i=1} {\overset {k}
\sum}} f_i c_1^{i_1}c_2^{i_2}\cdots c_{r+1}^{i_{r+1}} =
  \underset {i_1+\cdots +i_{r+1}=n}{\underset {i=1} {\overset {k}
\sum}} g_i c_1^{i_1}c_2^{i_2}\cdots c_{r+1}^{i_{r+1}} ~for~ some ~
g_i\in M.$$ This implies that $\phi_1(\widetilde{f_1-g_1},\ldots
,\widetilde {f_k-g_k})=0$.
Since $\phi_1$ is injective, it follows that $f_i-g_i\in N\subset M$
for all $i = 1, \ldots k$. Hence $f_i\in M ~for~ i=1,\ldots , k.$
\end{proof}

Now we prove (1) implies (3) in Theorem \ref{kk-thm} for two
dimensional Cohen-Macaulay rings.
\begin{theorem}\label{br-equality2}
Let $(R,\mathfrak m)$ be a two dimensional Cohen-Macaulay local ring
with infinite residue field and $M\subset F=R^r$ be a finitely
generated $R$-module with $\ell(F/M)<\infty$. If $\red_N(M) = 1$ for a
minimal reduction $N \subset M$, then for all $n\geq 0$,
$$\ell (\mathcal S_{n+1}(F)/\mathcal R_{n+1}(M))= \ell
(F/N){n+r+1\choose r+1}- \ell (M/N){n+r\choose r}.$$
In particular, if for any minimal reduction $N$ of $M$ $\red_N(M) =
1$, then $br_0(M) - br_1(M) = \ell(F/M)$ and $br_i(M) = 0$ for all $i
= 2, \ldots, r+1$.
\end{theorem}

\begin{proof}
Since $\red_N(M)$ is one, we have $\mathcal R_2(M)=N\mathcal R_1(M)$.
This implies $\mathcal R_{n+1}(M)=N\mathcal R_n(M)$ for all $n\ge 1$.
By induction, one can see that $\mathcal R_{n+1}(M)=M\mathcal R_n(N)$
for all $n\ge 0$. Consider the following short exact sequences of
R-modules with natural maps
$$0\longrightarrow \frac {\mathcal S_1(F)\mathcal R_n(N)}{\mathcal
R_1(M)\mathcal R_n(N)}\longrightarrow \frac{\mathcal S_{n+1}(F)}{\mathcal
R_{n+1}(M)} \longrightarrow \frac{\mathcal S_{n+1}(F)}{\mathcal
S_1(F)\mathcal R_n(N)} \longrightarrow 0,$$
$$0\longrightarrow \frac {\mathcal S_1(F)\mathcal R_n(N)}{\mathcal
R_{n+1}(N)}\longrightarrow \frac{\mathcal S_{n+1}(F)}{\mathcal R_{n+1}(N)}
\longrightarrow \frac{\mathcal S_{n+1}(F)}{\mathcal S_1(F)\mathcal R_n(N)}
\longrightarrow 0.$$
By additivity of the length function on short exact sequences, we get
$$\ell \left(\frac{\mathcal S_{n+1}(F)}{\mathcal R_{n+1}(M)}\right)  =
\ell\left(\frac{\mathcal S_{n+1}(F)}{\mathcal R_{n+1}(N)}\right) + \ell
\left(\frac {\mathcal S_1(F)\mathcal R_n(N)}{\mathcal R_1(M)\mathcal
R_n(N)}\right) - \ell \left(\frac {\mathcal S_1(F)\mathcal
R_n(N)}{\mathcal R_{n+1}(N)}\right).$$
Let $k={n+r\choose r}$. By Lemma \ref {lemmaKK}, $\left(\frac
{F}{M}\right)^k\cong \frac {F\mathcal R_n(N)}{M\mathcal R_n(N)}$ and
$\left(\frac {F}{N}\right)^k\cong \frac {F\mathcal R_n(N)}{\mathcal
  R_{n+1}(N)}$.
Hence $\ell (\frac {F\mathcal R_n(N)}{M\mathcal R_n(N)})= \ell
(F/M){n+r\choose r}$ and $\ell (\frac {F\mathcal R_n(N)}{\mathcal
  R_{n+1}(N)})= \ell (F/N){n+r\choose r}$.
Since $N$ is a parameter module, by \cite [Theorem 3.4]{BUV}, $\ell
(\mathcal S_{n+1}(F)/\mathcal R_{n+1}(N))=br_0(N){n+r+1\choose r+1}=br_0(M){n+r+1\choose r+1}$.
  \\Therefore
  \begin{eqnarray*}
    \ell \left(\frac{\mathcal S_{n+1}(F)}{\mathcal R_{n+1}(M)}\right)
    &=& br_0(M){n+r+1\choose r+1}+[\ell (F/M)-\ell (F/N)]{n+r\choose r}\\
    &=& br_0(M){n+r+1\choose r+1}-\ell(M/N){n+r\choose r}\\
    &=& \ell (F/N){n+r+1\choose r+1}-\ell(M/N){n+r\choose r}.
  \end{eqnarray*}
The second assertion now follows from the above equality.
\end{proof}
The main hurdle in proving a $d$-dimensional version of the above
theorem is in generalizing Theorem \ref{hyry-perfect}, which is not
known for modules $M$ with
$\widetilde{M}$ being a perfect matrix of size $r \times (d+r-1)$,
where $d = \dim R$.

\section{Main Result}
In this section, we prove an analogue of the Northcott inequality for
submodules of free modules over $2$-dimensional Cohen-Macaulay rings,
which have finite co-length.
W. V. Vasconcelos introduced the notion of Sally modules $S_J(I)$,
where $I$ is an ideal with a reduction $J,$ to study the interplay
between the depth properties of the blowup algebras and the properties
of the Hilbert-Samuel coefficients. The Sally module
$S_J(I)$ of $I$ with respect to $J$ is the $\mathcal R(J)$-module
defined by the following short exact sequence
$$0\rightarrow I\mathcal R(J)\rightarrow I\mathcal R(I)\rightarrow
S_J(I):={\underset {n\ge 0}\oplus}I^{n+1}/IJ^n\rightarrow0.$$
We refer the reader to \cite{WV} for basic properties of Sally
modules.
This definition can be extended to inclusion of graded algebras,
\cite{WV}.
As we have $\oplus_n\mathcal R_n(N)\subseteq \oplus_n \mathcal R_n(M)$ for
any reduction $N$ of $M$, we define the Sally module in an analogous
manner:

\begin{definition}
Let $(R,\mathfrak m)$ be a Noetherian local ring and $M \subset F =
R^r$ be a finitely generated $R$-module. Let $N\subset M$
be a $R$-submodule. Then Sally module of $M$ with respect to $N$
is defined as $ S_N(M):=\underset {n\ge 1}\oplus \frac {\mathcal
{R}_{n+1}(M)}{M\mathcal {R}_{n}(N)}.$
\end{definition}

We note that $S_N(M)$ is zero if and only if $\red_N(M)$ is at most one.
Note also that $\R(N)$ is a finitely generated standard graded algebra over $R$
and $S_N(M)$ is a finitely generated module over $\mathcal R(N)$.
Suppose $M \subset F=R^r$ is such that $\ell(F/M) < \infty$ and $N$
is a minimal reduction of $M$. Then the Hilbert function theory for
graded modules says that Hilbert function, $H(n)=\ell _R\left(\frac
{\mathcal {R}_{n+1}(M)}{M\mathcal {R}_{n}(N)}\right)$ is given by a
polynomial for $n \gg 0$ of degree equal to the dimension of $S_N(M)$.
Since $\m \R(N)\subset \mathfrak{p}$ for all $\mathfrak{p}\in Ass (S_N(M))$ it follows that $\dim S_N(M) \leq d+r-1$.
In the following Theorem we relate Hilbert function of $S_N(M)$ and Buchsbaum-Rim function of module $M$ in 2 dimensional Cohen-Macaulay ring.
As a consequence we obtain the Northcott inequality. The proof is
analogous to the corresponding results in  Section 2.1.2 of \cite {WV}.
\begin{theorem}
Let $(R,\mathfrak m)$ be a Cohen-Macaulay local ring of dimension $2$
with infinite residue field and $M\subseteq F=R^r$ with $\ell (F/M) <
\infty.$ Let the Buchsbaum-Rim polynomial corresponding to the
Buchsbaum-Rim function $BF(n)=\ell \left(\frac {\mathcal S_n(F)}{\mathcal
R_n(M)}\right)$ be given by
$$BP(n) = br_0(M){n+r \choose r+1}-br_1(M) {n+r-1\choose r}+\cdots +
(-1)^{r+1}br_{r+1} (M).$$
Suppose $N\subseteq M$ is a minimal reduction and $S=S_N(M)$ be the
corresponding Sally module, then for all $n\ge 0,$
$$BF(n)=br_0(M){n+r \choose r+1} +[\ell (F/M)-br_0(M)]{n+r-1\choose
r}-\ell (S_{n-1}).$$
\end{theorem}

\begin{proof}

Consider the following two short exact sequences of $R$-modules
$$0\longrightarrow \frac{M\mathcal R_{n-1}(N)}{\mathcal R_n(N)}\longrightarrow
\frac{\mathcal R_n(M)}{\mathcal R_n(N)}\longrightarrow \frac{\mathcal
R_n(M)}{M\mathcal R_{n-1}(N)}\longrightarrow 0,$$
$$0\longrightarrow \frac{M\mathcal R_{n-1}(N)}{\mathcal R_n(N)}\longrightarrow
\frac{F\mathcal R_{n-1}(N)}{\mathcal R_n(N)}\longrightarrow \frac{F\mathcal
  R_{n-1}(N)}{M\mathcal R_{n-1}(N)}\longrightarrow 0.$$
  Set $k={n+r\choose r}$.
By Lemma \ref {lemmaKK}, 
it follows that $\ell (\frac {F\mathcal R_n(N)}{M\mathcal R_n(N)})= \ell
(F/M){n+r\choose r}$ and $\ell (\frac {F\mathcal R_n(N)}{\mathcal R_{n+1}(N)})= \ell (F/N){n+r\choose r}$.
Therefore we have

    \begin{eqnarray*}
    BF(n)&=&\ell \left(\frac {\mathcal S_n(F)}{\mathcal R_n(M)}\right)\\
    &=& \ell \left(\frac {\mathcal S_n(F)}{\mathcal R_n(N)}\right)- \ell
	\left(\frac {\mathcal R_n(M)}{\mathcal R_n(N)}\right)\\
    &=& \ell \left(\frac {\mathcal S_n(F)}{\mathcal R_n(N)}\right)+\ell
	\left(\frac {F\mathcal R_{n-1}(N)}{M\mathcal R_{n-1}(N)}\right) - \ell
	\left(\frac {F\mathcal R_{n-1}(N)}{\mathcal R_n(N)}\right)-\ell
	\left(\frac {\mathcal R_n(M)}{M\mathcal R_{n-1}(N)}\right)\\
    &=& br_0(N){n+r\choose r+1}+\ell \left(\frac
	{F}{M}\right){n+r-1\choose r}\\
	& & -\ell\left(\frac
	{F}{N}\right){n+r-1\choose r}-\ell \left(\frac{\mathcal
	R_n(M)}{M\mathcal R_{n-1}(N)}\right)\\
    &=& br_0(M){n+r\choose r+1}+[\ell (F/M)-br_0(M)]{n+r-1\choose r}-\ell (S_{n-1}).
    \end{eqnarray*}
\end{proof}

We now derive the Northcott type inequality for the Buchsbaum-Rim
coefficients in $2$-dimensional Cohen-Macaulay local rings.

\begin{theorem}\label{northcott}
Let $(R,\m)$ be a Cohen-Macaulay local ring of dimension $2$, $M
\subset F=R^r$ be such that $\ell(F/M) < \infty.$ Then
$br_0(M)-br_1(M)\le \ell (F/M)$. If the reduction number of $M$ is at
most $1$, then the equality holds.
\end{theorem}
\begin{proof}
Let $BP(n)$ denote Buchsbaum-Rim polynomial of $M$. Then by the
previous theorem for $n\gg 0$ we get,
  \begin{eqnarray*}
    \ell (S_{n-1})&=&br_0(M){n+r \choose r+1} +[\ell (F/M)-br_0(M)]{n+r-1\choose r}-BP(n)\\
    &=&[\ell (F/M)-br_0(M)+br_1(M)]{n+r-1\choose
	r}-br_2(M){n+r-2\choose r-1}\\
	& & +\cdots+(-1)^{r}br_{r+1}.
  \end{eqnarray*}
This implies $\ell (F/M)-br_0(M)+br_1(M)$ is non-negative, i.e.,
$br_0(M)-br_1(M)\le \ell(F/M).$
\vskip 2mm
\noindent
If for a minimal reduction $N$ of $M$, $\red_N(M) \leq 1$, then
$S_N(M) = 0$ and consequently $\ell(F/M) - br_0(M) + br_1(M) = 0$,
i.e., $br_0(M) - br_1(M) = \ell(F/M)$.
\end{proof}

\section {Direct sum of ideals}
In this section we consider the modules $M$ which are direct sum of
several copies of an $\m$-primary ideal $I$. We explicitly compute
$br_0(M)$ and $br_1(M)$ in terms of $e_0(I)$ and $e_1(I)$. As a
consequence, we prove the Northcott inequality in this case.  We also
prove that in dimension $2$, the Northcott equality holds if and only
if the reduction number is at most $1$.

\begin{theorem}\label{newthm}
Let $(R,\mathfrak m)$ be a Cohen-Macaulay local ring of dimension $d
\geq 2$
and $I$ be an $\mathfrak m$-primary ideal. For $r\in \mathbb N$, set
$F=R^r$ and $M=I\oplus \cdots \oplus I$ (r times). Then $br_0(M) -
br_(M) \leq \ell(F/M)$.
\end{theorem}
\begin{proof}
Let $P_I(n) =
\displaystyle{\sum_{i=0}^d e_i{n+d-i-1 \choose d-i}}$
be the Hilbert-Samuel polynomial of $I$. Then by \cite[Theorem
2.5.2]{R}, the Buchsbaum-Rim polynomial is given by
\begin{eqnarray*}
 BP(n)
  &=& P_I(n){n+r-1\choose r-1}\\
  &=& [e_0{n+d-1\choose d}-e_1{n+d-2\choose d-1}+\cdots]{n+r-1\choose r-1}\\
  &=& e_0 \frac{(d+r-1)!}{d!(r-1)!}{n+d+r-2\choose d+r-1}\\
  & & -[e_0(d-1)\frac{(d+r-2)!}{d!(r-2)!}+e_1
	\frac{(d+r-2)!}{(d-1)!(r-1)!}]{n+d+r-3\choose d+r-2}+ \cdots
\end{eqnarray*}
Therefore, $br_0(M)=e_0{d+r-1\choose r-1}$ and
$br_1(M)=e_0(d-1){d+r-2\choose r-2}+e_1{d+r-2\choose r-1}$. We now
split the proof into two cases:

\vskip 2mm
\noindent
\textbf{Case 1: $d = 2$}
\vskip 1mm
In this case, we have
$br_0(M)=e_0{r+1\choose 2}$ and $br_1(M)=e_0{r\choose
2}+e_1r$. Hence $br_0(M)-br_1(M)=e_0r-e_1r\le r\ell(R/I)=\ell(F/M).$

\vskip 2mm
\noindent
\textbf{Case 2: $d\ge 3$}
\vskip 1mm
Let $r = 2$. We then have, $br_0(M) = e_0(d+1)$ and $br_1(M) =
e_0(d-1) + e_1 d$.  Therefore, $br_0(M) - br_1(M) = 2e_0 - de_1 =
2(e_0 - e_1) - (d-2) e_1 \leq 2\ell(R/I) = \ell(F/M)$. Note that in
this case, $br_0(M) - br_1(M) = \ell(F/M)$ if and only if $e_1 = 0$ if
and only if $I$ is a parameter ideal.

Now let $r\ge 3$. We then have,
\begin{eqnarray*}\label{ne_expression1}
  br_0(M)&-&br_1(M)-\ell(F/M)\\
  &=&e_0\left[{d+r-1\choose r-1}-(d-1){d+r-2\choose r-2}\right]-e_1{d+r-2\choose r-1}-r\ell(R/I).
\end{eqnarray*}

If $d=3$ and $r=3$, then the above expression becomes
\begin{eqnarray*}
  10e_0-8e_0-6e_1-3\ell(R/I)&=&2(e_0-e_1)-4e_1-3\ell(R/I)\\
  &\le& -4e_1-\ell(R/I)\le 0.
\end{eqnarray*}
Since $(R,\m)$ is Cohen-Macaulay, $e_1\ge 0$. Therefore, to prove the
Northcott inequality, it is enough to show that
\begin{eqnarray}\label {ne_expression}
  \left[{d+r-1\choose r-1}-(d-1){d+r-2\choose r-2}\right]e_0-r\ell(R/I)\le 0.
\end{eqnarray}
Considering the coefficient of $e_0$ in the above expression, we get
\begin{eqnarray*}
  {d+r-1\choose r-1}-(d-1){d+r-2\choose r-2}&=&{d+r-2 \choose r-2}\left[\frac{d+r-1}{r-1}- (d-1) \right]\\
  &=&{d+r-2 \choose r-2}\left[2-\frac{r-2}{r-1}d\right].
\end{eqnarray*}
It is a simple verification to see that this expression is
non-positive, and hence (\ref{ne_expression1}) holds, for $d = 3; r \geq 4$ and $d \geq 4; r \geq 3$. 
\end{proof}

Below we show that the direct sum of parameter ideal, in rank $2$, has
reduction number one.
\begin{proposition}\label{para}
Let $(R,\m)$ be a Cohen-Macaulay local ring of dimension $d \geq 2$, $I =
(a_1, \ldots, a_d)$ be a parameter ideal and $M = I \oplus I$. Then
the submodule $N$ of $M$ generated by the columns of the matrix
$\displaystyle{
\begin{bmatrix}
  a_1 & a_2 & \cdots & a_d & 0 \\
  0   & a_1 & \cdots & a_{d-1} & a_d
\end{bmatrix}
}$ is a minimal reduction of $M$ with $\red_N(M) = 1$.
\end{proposition}
\begin{proof}
Using the isomorphism $\R(M) \cong R[It_1, It_2]$, we move all the
computations to the bigraded Rees algebra. To prove the assertion, it
is enough to show that
\begin{eqnarray}\label{j-red}
I^2t_1^2 + I^2t_1t_2 + I^2t_2^2 & = & (a_1t_1,
a_2t_1+a_1t_2, \ldots, a_dt_1+a_{d-1}t_2,
a_dt_2)(It_1+It_2).
\end{eqnarray}
Set $L = (a_1t_1,
a_2t_1+a_1t_2, \ldots, a_dt_1+a_{d-1}t_2,
a_dt_2)(It_1+It_2).$
We show that for any $1 \leq i,j \leq d, ~ a_ia_jt_1^2, a_ia_jt_1t_2,
a_ia_jt_2^2$ belong to $L$. First note that for all $1 \leq i, j \leq
d$ the elements $a_1a_jt_1^2, a_1a_jt_1t_2, a_ia_dt_1t_2, a_ia_dt_2^2$ are all in
$L$.  Consider the following set of equations:
\begin{eqnarray*}
  a_ia_jt_1^2 & = & a_jt_1(a_it_1+a_{i-1}t_2) - a_ja_{i-1}t_1t_2 \\
  a_ja_{i-1}t_1t_2 & = & a_jt_2(a_{i-1}t_1+a_{i-2}t_2) -
  a_ja_{i-2}t_2^2 \\
  a_ja_{i-2}t_2^2 & = & a_{i-2}t_2(a_{j+1}t_1+a_jt_2) -
  a_{i-2}a_{j+1}t_1t_2 \\
  a_{i-2}a_{j+1}t_1t_2 & = &a_{i-2}t_1(a_{j+2}t_1+a_{j+1}t_2) -
  a_{i-2}a_{j+2}t_1^2.
\end{eqnarray*}
Then $a_ia_jt_1^2 \in L$ if and only if $a_{i-2}a_{j+2}t_1^2 \in L$.
If $i = 2$, the first equation itself will yield that $a_ia_jt_1^2 \in L$.
If $j = d-1$, then the third equation will yield that $a_ia_jt_1^2 \in
L$. If $i > 2$ and $j < d-1$, proceeding as above, one will hit an
element of the form $a_1a_jt_1^2, a_1a_jt_1t_2, a_ia_dt_1t_2$ or
$a_ia_dt_2^2$, which will imply that $a_ia_jt_1^2 \in L$. Similar
arguments will give us the other required inclusions. Hence $\red_N(M)
= 1$.
\end{proof}
\begin{corollary}\label{directsum}
Let $(R,\m)$ be a $d$-dimensional Cohen-Macaulay local ring,
$I$ be an $\m$-primary ideal and $M = I \oplus \cdots \oplus I$
($r$-times).
\begin{enumerate}
  \item If $d = 2$, then $br_0(M) - br_1(M) = \ell(F/M)$ if and only
	if $\red(M) = 1$.
  \item If $d \geq 3$, $r = 2$ and $br_0(M) - br_1(M) =
	\ell(F/M)$, then $\red(M) = 1$.
\end{enumerate}
\end{corollary}
\begin{proof}
(1) From the Case 1 in the above discussion preceding Proposition
\ref{para}, it follows that $br_0(M) - br_1(M) = \ell(F/M)$ if and only if
$e_0 - e_1 = \ell(R/I)$ if and only if $\red(I) \leq 1$ if and only if
$\red(M) = 1$, by Remark \ref{red}.
\vskip 2mm
\noindent
(2) From the Case 2 above, it follows that $br_0(M) - br_1(M) =
\ell(F/M)$ if and only if $I$ is a parameter ideal. Now, it follows
from the Proposition \ref{para} that
if $I$ is a parameter ideal, then
$I\oplus I$ has reduction number one.
\end{proof}

If the rank of $M$ is three, then an analogue Proposition \ref{para}
does not hold. Let $M = \m \oplus \m \oplus \m$, where $\m = (x,y,z)
\subset k[\![x,y,x]\!]$. Then it can be seen that the submodule $N$
generated by the columns of the matrix
$
\begin{bmatrix}
x & y & z & 0 & 0 \\
0 & x & y & z & 0 \\
0 & 0 & x & y & z
\end{bmatrix}
$
is a minimal reduction of $M$ with $\red_N(M) = 2$. The idea of
getting minimal reduction of the above form comes from the work of J.
-C. Liu, \cite{liu}.
\vskip 2mm
\noindent
\begin{example}\label{ex1}
Let $R=k[\![X,Y]\!], I=(X^3,X^2Y^4,XY^5,Y^7), J=(X^3,Y^7)$. Then $R$ is
a $2$-dimensional regular local ring and $J$ is a minimal reduction of
$I$ with reduction number $2$. It can be easily seen that
$P_I(n)=21{n+1\choose 2}-6{n\choose 1}+1$. Set
$F=R\oplus R, M=I\oplus I$. Then again using \cite[Theorem 2.5.2]{R},
we get $br_0=63$ and $br_1=33$.
Therefore $br_0(M)-br_1(M) = 30 < 32 = \ell(F/M).$ Let $N$ be the
submodule generated by the columns of
$
\begin{bmatrix}
X^3 & Y^7 & 0 \\
0 & X^3 & Y^7
\end{bmatrix}.$
Then, it can be seen that $N$ is a minimal reduction of $M$ with
$\red_N(M) = 2$.
\end{example}
As in the case of ideals,
the example below shows that the Cohen-Macaulayness of the Rees
algebra alone need not necessarily imply that $br_0(M) - br_1(M) =
\ell(F/M)$ if $\dim R \geq 3$.

\begin{example}
  Let $R=k[\![X,Y,Z]\!]$, $I=(X^3,X^2Y^2,Y^3,Z^4)$ and $M=I\oplus I$. It can be verified that $\R (M)\cong R[It_1,It_2]$ is Cohen-Macaulay. So by \cite [Theorem 6.1]{hyry}, $BF(n)=BP(n)$ for all $n\in \mathbb{N}$. The Buchsbaum-Rim polynomial can be computed as
  $$BP(n)=144{n+3\choose 4}-84{n+2\choose 3}+4{n+1\choose 2}.$$
  Therefore $br_0(M)-br_1(M)= 60 < 64 = \ell(F/M).$
\end{example}
We conclude the article with a question:
\begin{question}
Let $(R,\m)$ be a Cohen-Macaulay local ring of dimension $d > 2$ and $M
\subset F = R^r$ be such that $\ell(F/M) < \infty.$ Then is $br_0(M) -
br_1(M) \leq \ell(F/M)$? Does the equality $br_0(M) - br_1(M) =
\ell(F/M)$ hold if and only if
$\red_N(M) = 1$ for some (any) minimal reduction $N$ of $M$?
\end{question}

\end{document}